\documentclass[11pt,english]{article}
\usepackage[T1]{fontenc}
\usepackage[latin9]{inputenc}
\usepackage{verbatim}
\usepackage{amsmath}
\usepackage{amsthm}
\usepackage{amssymb}
\usepackage{graphicx}

\makeatletter
\theoremstyle{plain}
\newtheorem{thm}{\protect\theoremname}
  \theoremstyle{plain}
  \newtheorem{prop}[thm]{\protect\propositionname}
  \theoremstyle{plain}
  \newtheorem{lem}[thm]{\protect\lemmaname}
  \theoremstyle{plain}
  \newtheorem{question}[thm]{\protect\questionname}

\usepackage{fullpage}
\usepackage{hyperref}
\hypersetup{
  colorlinks   = true, 
  urlcolor     = red, 
  linkcolor    = blue, 
  citecolor   = blue 
}
\date{}
\@ifundefined{showcaptionsetup}{}{%
 \PassOptionsToPackage{caption=false}{subfig}}
\usepackage{subfig}
\makeatother

\usepackage{babel}
  \providecommand{\lemmaname}{Lemma}
  \providecommand{\propositionname}{Proposition}
  \providecommand{\questionname}{Question}
\providecommand{\theoremname}{Theorem}

\begin{document}
\global\long\def\goinf{\rightarrow\infty}
\global\long\def\gozero{\rightarrow0}
\global\long\def\bra{\langle}
\global\long\def\ket{\rangle}
\global\long\def\union{\cup}
\global\long\def\intersect{\cap}
\global\long\def\abs#1{\left|#1\right|}
\global\long\def\norm#1{\left\Vert #1\right\Vert }
\global\long\def\floor#1{\left\lfloor #1\right\rfloor }
\global\long\def\ceil#1{\left\lceil #1\right\rceil }
\global\long\def\expect{\mathbb{E}}
\global\long\def\e{\mathbb{E}}
\global\long\def\r{\mathbb{R}}
\global\long\def\n{\mathbb{N}}
\global\long\def\q{\mathbb{Q}}
\global\long\def\c{\mathbb{C}}
\global\long\def\p{\mathbb{P}}
\global\long\def\z{\mathbb{Z}}
\global\long\def\grad{\nabla}
\global\long\def\t{^{\prime}}
\global\long\def\all{\forall}
\global\long\def\eps{\varepsilon}
\global\long\def\quadvar#1{V_{2}^{\pi}\left(#1\right)}
\global\long\def\cross{\times}
\global\long\def\del{\nabla}
\global\long\def\parx#1{\frac{\partial#1}{\partial x}}
\global\long\def\pary#1{\frac{\partial#1}{\partial y}}
\global\long\def\parz#1{\frac{\partial#1}{\partial z}}
\global\long\def\part#1{\frac{\partial#1}{\partial t}}
\global\long\def\partheta#1{\frac{\partial#1}{\partial\theta}}
\global\long\def\parr#1{\frac{\partial#1}{\partial r}}
\global\long\def\curl{\nabla\times}
\global\long\def\rotor{\nabla\times}
\global\long\def\one{\mathbf{1}}
\global\long\def\Hom{\text{Hom}}
\global\long\def\pr#1{\text{Pr}\left[#1\right]}
\global\long\def\almost{\mathbf{\approx}}
\global\long\def\tr{\text{Tr}}
\global\long\def\var{\text{Var}}
\global\long\def\onenorm#1{\left\Vert #1\right\Vert _{1}}
\global\long\def\twonorm#1{\left\Vert #1\right\Vert _{2}}
\global\long\def\Inj{\mathfrak{Inj}}
\global\long\def\inj{\mathsf{inj}}
\global\long\def\discrete{\mathcal{C}_{n}}
\global\long\def\contin{\overline{\mathcal{C}_{n}}}
\global\long\def\lone{\mathrm{L}^{1}}
\global\long\def\ltwo{\mathrm{L}^{2}}
\global\long\def\elone{\mathrm{L}^{1}}
\global\long\def\eltwo{\mathrm{L}^{2}}
\global\long\def\Inf{\text{Inf}}

\global\long\def\g{\mathfrak{{\cal G}}}
\global\long\def\f{\mathfrak{{\cal F}}}
\global\long\def\tensor{\otimes}

\title{A conformal Skorokhod embedding}

\author{Renan Gross\thanks{Weizmann Institute of Science. Email: renan.gross@weizmann.ac.il.
Supported by the Adams Fellowship Program of the Israel Academy of
Sciences and Humanities.}}
\maketitle
\begin{abstract}
Start a planar Brownian motion and let it run until it hits some given
barrier. We show that the barrier may be crafted so that the $x$
coordinate at the hitting time has any prescribed centered distribution
with finite variance. This provides a new, complex-analytic proof
of the Skorokhod embedding theorem. Our method is constructive and
can give an explicit description of the barrier. 
\end{abstract}

\section{Introduction}

The Skorokhod embedding problem asks the following: Given a Brownian
motion $X_{t}$ and a probability distribution $\mu$ with expectation
zero and finite variance, find a stopping time $T$ so that $X_{t}\sim\mu$
and $\e T<\infty$.

There have been numerous solutions to this formulation and to several
variations and generalizations over the years. See Ob\l ój's extensive
survey \cite{obloj_skorkhod_embedding_and_its_offspring} for a detailed
account of the problem, its characteristics and applications.

One important solution is given by Root \cite{root_existence_of_certain_stopping_times_on_Brownian_motion},
who sets $T$ as the first time that the graph $\left(t,X_{t}\right)$
of the Brownian motion hits some barrier $\Omega\subseteq\r^{+}\times\r$.
Root's solution does not use any additional randomness. Finding out
the barrier $\Omega$, however, is often a difficult task, and not
many explicit solutions are known (but see \cite{gassiat_oberhauser_reis_root_barrier_viscosity}
for constructions relying on solutions to PDEs).

In this paper, we present a new solution to the Skorokhod embedding
problem. Our method is similar to Root's, in that the stopping time
is the first hitting time of some barrier by a Brownian motion. The
method requires additional randomness in the form of another independent
Brownian motion, but can offer tractable analytic expressions for
calculating the shape of the barrier explicitly. 

For a domain $\Omega\subseteq\r^{2}$ and a planar Brownian motion
$X_{t}=\left(X_{t}^{\left(1\right)},X_{t}^{\left(2\right)}\right)$
with $X_{0}\in\Omega$, let $T\left(X_{t},\Omega\right)$ be the first
time that $X_{t}$ exits the domain $\Omega$, 
\[
T\left(X_{t},\Omega\right)=\inf\left\{ t>0\mid X_{t}\notin\Omega\right\} .
\]
\begin{thm}
\label{thm:conformal_skorokhod_main_theorem}Let $\mu$ be a probability
distribution on $\r$ with zero expectation and finite variance. There
exists a simply connected domain $\Omega\subseteq\r^{2}$ containing
the origin such that if $Y_{t}$ is a standard planar Brownian, then
$Y_{T\left(Y_{t},\Omega\right)}^{\left(1\right)}$ has distribution
$\mu$. 
\end{thm}

The proof of Theorem \ref{thm:conformal_skorokhod_main_theorem} is
given in the next section. Section \ref{sec:properties_and_examples}
gives properties and examples of $\Omega$. Finally, in Section \ref{sec:open_questions}
we exhibit some open questions. 

\section{\label{sec:proof_of_theorem}Proof of theorem}
\begin{proof}
We identify $\r^{2}$ with the complex plane $\c$. Denote the open
unit disc by $D=\left\{ z\in\c\mid\abs z<1\right\} $ and the unit
circumference by $\partial D=\left\{ z\in\c\mid\abs z=1\right\} $.

Denote by $F:\r\to\left[0,1\right]$ the cumulative distribution function
of $\mu$, and by $G:\left(0,1\right)\to\r$ its generalized inverse:
\[
G\left(y\right)=\inf\left\{ x\in\r\mid F\left(x\right)\geq y\right\} .
\]
Observe that both $F$ and $G$ are monotone increasing. Despite the
fact that $G$ may diverge at $0$ and $1$, it is still in $\ltwo$:
If $X$ is a uniform random variable on $\left[0,1\right]$ then $G\left(X\right)$
distributes as $\mu$, and so 
\begin{align*}
\int_{0}^{1}G^{2}\left(y\right)dy & =\var\mu.
\end{align*}
Define the function $\varphi:\left(-\pi,\pi\right)\backslash\left\{ 0\right\} \to\r$
by 
\begin{equation}
\varphi\left(\theta\right)=G\left(\frac{\abs{\theta}}{\pi}\right).\label{eq:definition_of_phi}
\end{equation}
The function $\varphi$ may be viewed as a periodic function on $\r$
with period $2\pi$, with the possibility of it diverging at integer
multiples of $\pi$. It is an even function in $\ltwo$, and thus
has a Fourier series representation containing only cosines. Denote
the $n$-th Fourier coefficient of $\varphi$ by $\hat{\varphi}\left(n\right)$,
and recall that by Carleson's theorem \cite{carleson_pointwise_convergence}
for an $\ltwo$ function, the Fourier representation $\sum_{n=0}^{\infty}\hat{\varphi}\left(n\right)\cos\left(\theta n\right)$
agrees with $\varphi$ almost everywhere. 

Let $\psi:D\to\c$ be the complex function defined by 
\begin{equation}
\psi\left(z\right)=\sum_{n=0}^{\infty}\hat{\varphi}\left(n\right)z^{n}.\label{eq:definition_of_psi}
\end{equation}
The function $\psi$ may not necessarily be defined for points on
the unit circle, but its real part is well-defined for almost all
such points. Indeed, let $z=e^{i\theta}$ with $\theta\in\left[-\pi,\pi\right]$,
then
\begin{align*}
\mathrm{Re}\psi\left(z\right) & =\mathrm{Re}\psi\left(e^{i\theta}\right)\\
 & =\mathrm{Re}\sum_{n=0}^{\infty}\hat{\varphi}\left(n\right)e^{i\theta n}\\
 & =\sum_{n=0}^{\infty}\hat{\varphi}\left(n\right)\cos\theta n.
\end{align*}
So on the unit circle, the real part of $\psi\left(e^{i\theta}\right)$
agrees with $\varphi\left(\theta\right)$ almost everywhere. 

Since $\lim_{n\to\infty}\hat{\varphi}\left(n\right)=0$ by the Riemann-Lebesgue
lemma, the function $\psi$ is analytic (and non constant) in the
open disc $D$, and so the image $\Omega=\psi\left(D\right)$ is some
connected domain in $\c$. As all the coefficients of $\psi$ are
real, this domain is symmetric to conjugation: $\psi\left(\overline{z}\right)=\overline{\psi\left(z\right)}$,
i.e it is symmetric to reflection about the $x$ axis. 
\begin{prop}
$\psi$ is one-to-one in the unit disc $D$. 
\end{prop}

\begin{proof}
The proof relies on the following theorem, which can be found in \cite[chapter VIII.3]{theodore_gamelin_complex_analysis}:
\begin{thm}
Let $\left\{ f_{k}\left(z\right)\right\} _{k=1}^{\infty}$ be a sequence
of one-to-one analytic functions on a domain $D$ that converge uniformly
on every compact subset of $D$ to a function $f$. Then $f$ is either
one-to-one or constant. %
\end{thm}

All we have to do then is find a sequence of one-to-one functions
$\psi_{k}$ which converge to $\psi$ uniformly on every compact subset
of $D$. Let $\left\{ G_{k}:\left[0,1\right]\to\r\right\} _{k=1}^{\infty}$
be a sequence of bounded, twice differentiable, strictly increasing
functions satisfying
\begin{enumerate}
\item $G_{k}\to G$ in $\lone$.
\item \label{enu:approximations_are_in_l1}The first and second derivatives
of $G_{k}^{\left(s\right)}$ are in $\lone$. 
\item \label{enu:approximations_have_zero_derivative}$G_{k}'\left(0\right)=G_{k}'\left(1\right)=0$
and $G_{k}''\left(0\right)=G_{k}''\left(1\right)=0$.
\end{enumerate}
Such a sequence may be found, for example, by taking step-function
approximations of $G$ on increasingly finer partitions of $\left(0,1\right)$,
smoothingly interpolating the jump discontinuities, and adding some
small strictly increasing smooth function. 

For each $G_{k}$, define the corresponding $\varphi_{k}$ and $\psi_{k}$
as in equations (\ref{eq:definition_of_phi}) and (\ref{eq:definition_of_psi}).
By properties (\ref{enu:approximations_are_in_l1}) and (\ref{enu:approximations_have_zero_derivative})
above, the $\varphi_{k}$ are also smooth and their derivatives are
all in $\lone$. Using the fact that if a function $f$ is $s$-times
differentiable and $f^{\left(s\right)}\in\lone$ then $\abs{\hat{f}\left(n\right)}\leq\norm{f^{\left(s\right)}}_{1}/n^{s}$,
we get that the Fourier coefficients $\left\{ \hat{\varphi}_{k}\left(n\right)\right\} _{n=0}^{\infty}$
are absolutely convergent, and so $\psi_{k}$ can be extended to the
closed disc $\overline{D}$. This allows us to look at the image of
the unit circle $\partial D$ under $\psi_{k}$.

Parameterize the circle by $\gamma\left(\theta\right)=e^{i\theta}$
for $\theta\in\left[-\pi,\pi\right]$. As $\theta$ increases from
$-\pi$ to $0$, $\mathrm{Re}\psi_{k}\left(\gamma\left(\theta\right)\right)=\varphi_{k}\left(\theta\right)$
strictly decreases from $G_{k}\left(1\right)$ to $G_{k}\left(0\right)$,
and as $\theta$ increases from $0$ to $\pi$, $\mathrm{Re}\psi_{k}\left(\gamma\left(\theta\right)\right)$
strictly increases from $G_{k}\left(0\right)$ to $G_{k}\left(1\right)$.
Further, using the Hilbert transform equation (\ref{eq:hilbert_transform_definition})
in Section \ref{sec:properties_and_examples}, it can be directly
calculated that $\mathrm{Im}\psi_{k}\left(\gamma\left(\theta\right)\right)$
is always positive for $\theta\in\left(-\pi,0\right)$ and always
negative for $\theta\in\left(0,\pi\right)$. Thus, the image $\psi_{k}\left(\partial D\right)$
is a simple loop. Since the preimage of $\partial\psi_{k}\left(D\right)$
is a subset of $\partial D$ (this is a consequence of the maximum
principle: If an interior point $z\in D$ were mapped to $\partial\psi_{k}\left(D\right)$,
we could obtain a local maximum for $\abs{\psi_{k}}$), we have that
$\psi_{k}\left(\partial D\right)=\partial\psi_{k}\left(D\right)$.
So for every $x\in\psi_{k}\left(D\right)$, $\psi_{k}\left(\partial D\right)$
winds around $x$ exactly once. By Cauchy's argument principle, this
means that $x$ has exactly one preimage in $D$ under $\psi_{k}$,
i.e $\psi_{k}$ is one-to-one. 

All that remains is to show that $\left\{ \psi_{k}\right\} _{k=1}^{\infty}$
converge uniformly to $\psi$ on every compact subset $A\subseteq D$
. To see this, note that $A$ is contained in some closed disc of
radius $\rho<1$. Denoting $z=re^{i\theta}\in A$ with $0\leq r\leq\rho$,
we have 
\begin{align*}
\abs{\psi\left(z\right)-\psi_{k}\left(z\right)} & =\abs{\sum_{n=0}^{\infty}\left(\hat{\varphi}\left(n\right)-\hat{\varphi}_{k}\left(n\right)\right)r^{n}e^{in\theta}}\\
 & \leq\sum_{n=0}^{\infty}\abs{\left(\hat{\varphi}\left(n\right)-\hat{\varphi}_{k}\left(n\right)\right)r^{n}e^{in\theta}}\\
 & \leq\sup_{n}\abs{\left(\hat{\varphi}\left(n\right)-\hat{\varphi}_{k}\left(n\right)\right)}\sum_{n=0}^{\infty}r^{n}.
\end{align*}
The sum $\sum_{n=0}^{\infty}r^{n}$ converges, and the supremum converges
to $0$ uniformly in $k$ since for each $n\in\n$ we have 
\begin{align*}
\abs{\hat{\varphi}\left(n\right)-\hat{\varphi}_{k}\left(n\right)} & =\abs{\frac{1}{2\pi}\int_{-\pi}^{\pi}\left(\varphi\left(\theta\right)-\varphi_{k}\left(\theta\right)\right)\cos\left(n\theta\right)d\theta}\\
 & \leq\frac{1}{2\pi}\int_{-\pi}^{\pi}\abs{\varphi\left(\theta\right)-\varphi_{k}\left(\theta\right)}d\theta,
\end{align*}
and this integral converges to $0$ since $\varphi_{k}$ converges
to $\varphi$ in $\lone$. 
\end{proof}
Having established that $\psi$ is one-to-one, we can invoke the conformal
invariance of Brownian motion, which can be found in \cite[Theorem 7.20]{peres_morters_brownian_motion}:
\begin{thm}
Let $U$ be a domain in the complex plain, $x\in U$ and $f:U\to\Omega$
be analytic. Let $\left\{ X_{t}\mid t\geq0\right\} $ be a planar
Brownian motion started in $x$. Then the process $\left\{ f\left(X_{t}\right)\mid0\leq t\leq T\left(X_{t},U\right)\right\} $
is a time-changed Brownian motion, i.e there exists a planar Brownian
motion $\left\{ Y_{t}\mid t\geq0\right\} $ such that for any $t\in\left[0,T\left(X_{t},U\right)\right)$,
\[
f\left(X_{t}\right)=Y_{\zeta\left(t\right)}
\]
where 
\[
\zeta\left(t\right)=\int_{0}^{t}\abs{f'\left(X_{s}\right)}^{2}ds.
\]
If $f$ is one-to-one then $\zeta\left(T\left(X_{t},U\right)\right)$
is the first exit time from $\Omega$ by $\left\{ Y\mid t\geq0\right\} $.
\end{thm}

Thus, the process $\psi\left(X_{t}\right)$ is a time-changed Brownian
motion, i.e there exists a planar Brownian motion $Y_{t}$ and a monotone
function $\zeta\left(t\right)$ such that $\psi\left(X_{t}\right)=Y_{\zeta\left(t\right)}$.
The domain $\Omega$ has the properties that we are looking for: 
\begin{enumerate}
\item Upon hitting the boundary, $Y_{t}^{\left(1\right)}$ distributes as
$Y_{T\left(Y_{t},\Omega\right)}^{\left(1\right)}\sim\mu$: The position
of $X_{T\left(X_{t},D\right)}$ is uniform on the unit circle and
the real part of $\psi\left(\partial D\right)$ is made of two (reflected)
copies of $G=F^{-1}$, so by the inverse transform sampling method
we get that $\mathrm{Re}Y_{T\left(Y_{t},\Omega\right)}$ distributes
as $\mu$.
\item $Y_{t}$ hits the boundary $\partial\Omega$ in finite time. This
is an immediate consequence of the following lemma:
\begin{lem}[Lemma 1.1 in \cite{banuelos_carrol_brownian_motion_drum}]
Suppose that $f\left(w\right)=\sum_{n=0}^{\infty}a_{n}w^{n}$ is
a conformal mapping from the unit disc $D$ onto a domain $\Omega$
with $f\left(0\right)=z$. Then 
\begin{align}
\e_{z}\left(T\left(Y_{t},\Omega\right)\right) & =\frac{1}{2}\sum_{n=1}^{\infty}\abs{a_{n}}^{2}.\label{eq:formula_for_expected_stopping_time}
\end{align}
\end{lem}

Since $\varphi$ is in $\ltwo$, the right hand side of equation (\ref{eq:formula_for_expected_stopping_time})
is finite. Since $\left(Y_{t}^{\left(1\right)}\right)^{2}-t$ is a
martingale with expectation $0$, we get by the optional stopping
theorem that 
\[
\e\left(Y_{T\left(Y_{t},\Omega\right)}^{\left(1\right)}\right)^{2}-\e T\left(Y_{t},\Omega\right)=0,
\]
implying $\e T\left(Y_{t},\Omega\right)=\var\mu$.
\item The process $Y_{t}$ starts at $0$, i.e $Y_{0}=\psi\left(0\right)=0$:
The imaginary component is $0$ since $\psi$ is symmetric to complex
conjugation, and the real component is $0$ since $Y_{t}$ is a martingale
and $\mathrm{Re}Y_{t}$ distributes as $\mu$, so that $\mathrm{Re}Y_{0}=\int_{\r}xd\mu=0$. 
\item $\Omega$ is simply connected by Brouwer's invariance of domain theorem,
which states that:
\begin{thm}[Theorem 1 in chapter 1 in \cite{spivak_comprehensive_introduction_to_differential_geometry}]
If $U\subseteq\r^{n}$ is open and $f:U\to\r^{n}$ is one-to-one
and continuous, then f is a homeomorphism.
\end{thm}

\end{enumerate}
\end{proof}

\section{\label{sec:properties_and_examples}Properties and examples}

In this section, we will see examples of domains $\Omega$ for various
distributions. Numeric approximations of $\Omega$ for Bernoulli,
Gaussian and Cantor distributions, as well as some distribution on
the natural numbers, can be found in Figure \ref{fig:examples_of_domains}. 

\begin{figure}
\begin{centering}
\subfloat[$\pm1$ Bernoulli distribution]{\includegraphics[scale=0.45]{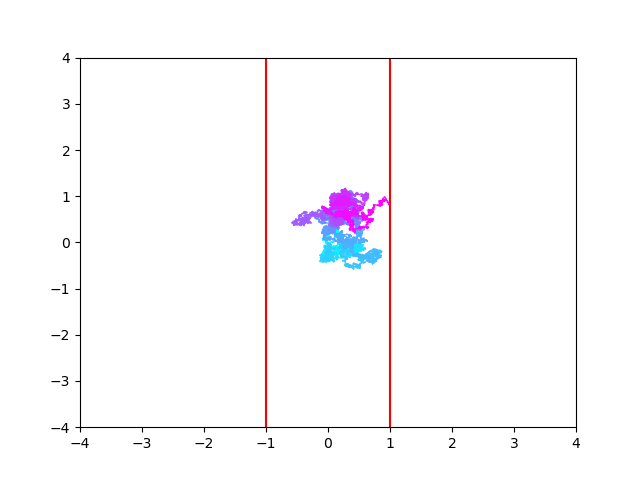}

}\subfloat[{Uniform distribution on $\left[-1,1\right]$\label{fig:uniform_distribution_boundary}}]{\includegraphics[scale=0.45]{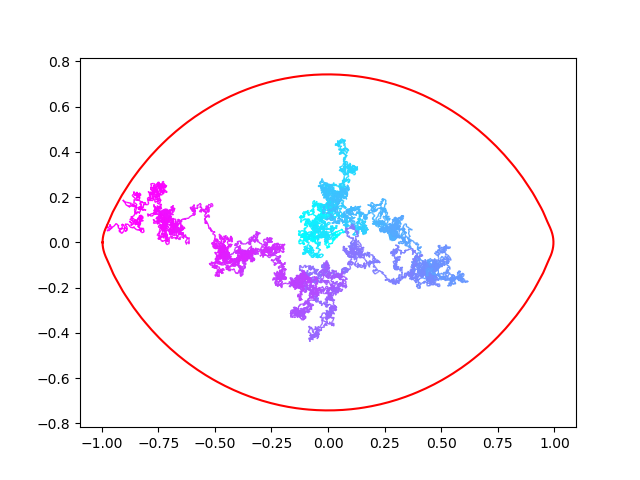}

}
\par\end{centering}
\begin{centering}
\subfloat[$\mu\left(\left\{ k\right\} \right)=\frac{1}{2^{k+1}}$ for integer
$k\geq0$. Here the Brownian does not start at $0$, since $\mu$
is not centered.]{\includegraphics[scale=0.45]{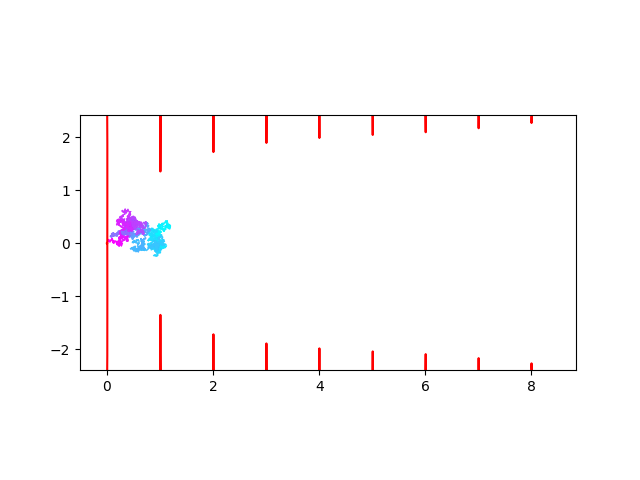}

}
\par\end{centering}
\begin{centering}
\subfloat[$N\left(0,1\right)$ Gaussian distribution]{\includegraphics[scale=0.45]{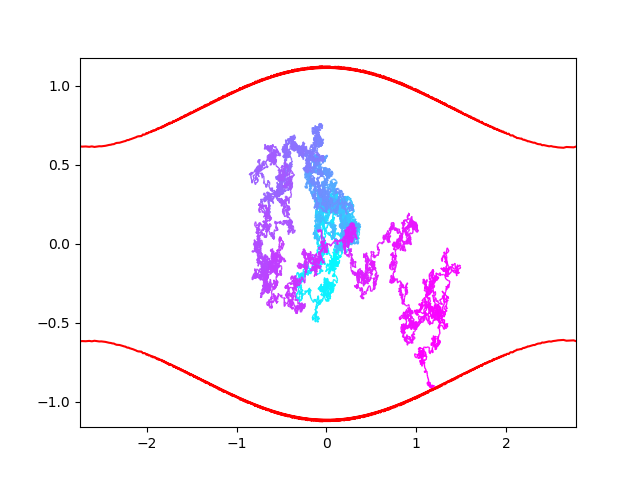}

}\subfloat[Cantor distribution. The fractal boundary is approximated by truncating
the very-slowly-converging Fourier series. All vertical line segments
should be extend to infinity.]{\includegraphics[scale=0.45]{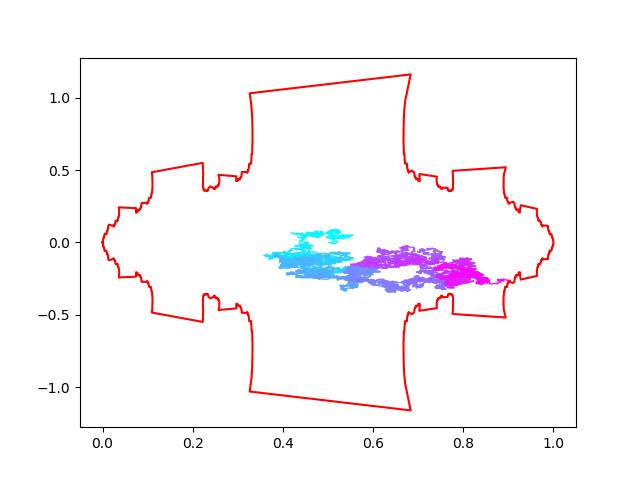}

}
\par\end{centering}
\caption{Examples of domains $\Omega$ for various distributions.\label{fig:examples_of_domains}}
\end{figure}

Some easy properties of $\Omega$ can immediately be gathered from
$\mu$:
\begin{enumerate}
\item \label{enu:domain_properties_atom}If $x$ is an atom of $\mu$, i.e
$\mu\left(\left\{ x\right\} \right)>0$, then $\partial\Omega$ contains
a straight line segment at $x$, i.e $\left\{ x\right\} \times\left[a,b\right]\subseteq\partial\Omega$
for some numbers $a,b$. 
\item \label{enu:domain_properties_no_interval}If $\left(a,b\right)$ is
an interval with $\mu\left(\left(a,b\right)\right)=0$ then $\Omega$
contains the infinite rectangle $\left(a,b\right)\times\left(-\infty,\infty\right)$.
\item \label{enu:domain_properties_infinite_x}If the support of $\mu$
is infinite, then the projection of $\Omega$ onto the $x$ axis is
unbounded.
\end{enumerate}
One can describe $\Omega$ by looking at its boundary $\partial\Omega$,
and in general it is useful to do this by considering how the curve
$\gamma\left(\theta\right)=\psi\left(e^{i\theta}\right)$ behaves
as $\theta$ traverses from $-\pi$ to $\pi$. Of course, for some
distributions both the real part and the imaginary part of $\psi\left(e^{i\theta}\right)$
can be infinite, in the latter case even countably many times, so
$\gamma\left(\theta\right)$ really sits in the Riemann sphere.

When calculating $\psi\left(e^{i\theta}\right)$, the real part is
already known - it is just equal to $\varphi\left(\theta\right)$
- so all that is left to do is to find the imaginary part. The relation
between the real part and the imaginary part of $\psi\left(e^{i\theta}\right)$
is given by replacing all cosines with sines in the series expansion
of $\psi$:
\begin{align*}
\mathrm{Re}\psi\left(e^{i\theta}\right) & =\sum_{n=0}^{\infty}\hat{\varphi}\left(n\right)\cos\theta n,\\
\mathrm{Im}\psi\left(e^{i\theta}\right) & =\sum_{n=1}^{\infty}\hat{\varphi}\left(n\right)\sin\theta n.
\end{align*}
Alternatively, the imaginary part is obtained from the real part by
the Hilbert transform operator $H$, which, for a function $u:\r\to\r$,
is given by 
\begin{equation}
\left(Hu\right)\left(x\right)=\frac{1}{\pi}PV\int_{-\infty}^{\infty}\frac{u\left(\tau\right)}{t-\tau}d\tau,\label{eq:hilbert_transform_definition}
\end{equation}
where $PV$ is Cauchy's principal value. One way to see that this
is true is to note that the Hilbert transform is linear and maps the
function $\cos nx$ to $\sin nx$ for positive integers $n$. For
a comprehensive source on the Hilbert transform, see \cite{king_hilbert_transforms}.

There are therefore two straightforward ways to go about studying
$\mathrm{Im}\psi$. The first is to inspect the series $\sum_{n=1}^{\infty}\hat{\varphi}\left(n\right)\sin\theta n$
directly, and the second is to consider the Hilbert transform of $\varphi$.
In the next two subsections, we will see both approaches. 

\subsection{Discrete distributions}

Let $\mu$ be an atomic distribution, supported on a (possibly infinite)
discrete set of values $\left\{ x_{i}\right\} _{i}$, which obtains
$x_{i}$ with probability $p_{i}$: 
\[
\mu\left(x\right)=\sum_{i}p_{i}\delta_{x_{i}}\left(x\right).
\]
The cumulative distribution function is then just a sum of step functions,
\[
F\left(x\right)=\sum_{i}p_{i}\one_{x\geq x_{i}},
\]
and so is the inverse $G\left(y\right)$ and the symmetric $\varphi_{\mu}\left(\theta\right)$.
We write
\[
\varphi\left(\theta\right)=\sum_{i}\alpha_{i}\one_{\abs{\theta}\geq\theta_{i}}
\]
for some weights $\alpha_{i}$ and thresholds $\theta_{i}$ (these
can be calculated explicitly from $p_{i}$ and $x_{i}$, but we omit
the calculations for brevity). 

As noted in items (\ref{enu:domain_properties_atom}) and (\ref{enu:domain_properties_no_interval})
above, the boundary $\partial\Omega$ consists of infinite rays of
the form $\left\{ x_{i}\right\} \times\left(y_{i},\infty\right]$
and $\left\{ x_{i}\right\} \times\left(-y_{i},\infty\right]$ for
some values $y_{i}\geq0$. If there is an extremal value $x=\min_{i}\left\{ x_{i}\right\} $
or $x=\max_{i}\left\{ x_{i}\right\} $ among the $x_{i}$'s, then
clearly $\partial\Omega$ contains the infinite line $\left\{ x\right\} \times\left(-\infty,\infty\right)$. 

As the $y_{i}$'s give a complete characterization of $\Omega$, calculating
them may be of interest.
\begin{thm}
Each $y_{i}$ is given by 
\[
y_{i}=\frac{1}{\pi}\sum_{j}\alpha_{j}\log\abs{\frac{\sin\left(\frac{\theta_{j}}{2}-\frac{x}{2}\right)}{\sin\left(\frac{\theta_{j}}{2}+\frac{x}{2}\right)}},
\]
where $x$ is a solution to the equation 
\[
\sum_{i}\alpha_{i}\left(\cot\left(\frac{\theta_{i}}{2}-\frac{x}{2}\right)+\cot\left(\frac{\theta_{i}}{2}+\frac{x}{2}\right)\right)=0.
\]
\end{thm}

\begin{proof}
Although $\mathrm{Im}\psi\left(e^{i\theta}\right)$ may be discontinuous
and even infinite, it is differentiable almost everywhere; in fact,
by section III.2.10 in \cite{katznelson_introduction_to_harmonic_analysis},
if $\mathrm{Re}\psi\left(e^{i\theta}\right)$ is constant in an interval,
then $\mathrm{Im}\psi\left(e^{i\theta}\right)$ is analytic there.
This means that $\mathrm{Im}\psi\left(e^{i\theta}\right)$ is piecewise
analytic, since $\mathrm{Re}\psi\left(e^{i\theta}\right)$ is piecewise
constant except for a discrete set of jumps at $\left\{ x_{i}\right\} _{i}$.

The $y_{i}$'s can therefore be calculated by finding the local minima
and maxima of $\mathrm{Im}\psi\left(e^{i\theta}\right)$, and these
in turn are obtained by differentiation. 

Our first step is to calculate $\mathrm{Im}\psi\left(e^{i\theta}\right)$.
As noted at the beginning of the section, the relation between $\mathrm{Re}\psi\left(e^{i\theta}\right)$
and $\mathrm{Im}\psi\left(e^{i\theta}\right)$ is given by the Hilbert
transform $H$. Since $H$ is a linear operator, we first compute
the Hilbert transform of a single step function.
\begin{lem}
\label{lem:hilbert_transform_of_step}Let $u\left(x\right)$ be a
periodic step function, i.e
\[
u\left(x\right)=\begin{cases}
0 & 0<\abs x<\theta_{0}\\
1 & \theta_{0}\leq\abs x\leq\pi
\end{cases}
\]
for some $0\leq\theta_{0}\leq\pi$, with $u\left(x\right)=u\left(x+2\pi\right)$.
Then 
\begin{equation}
\left(Hu\right)\left(x\right)=\frac{1}{\pi}\log\abs{\frac{\sin\left(\frac{\theta_{0}}{2}-\frac{x}{2}\right)}{\sin\left(\frac{\theta_{0}}{2}+\frac{x}{2}\right)}}.\label{eq:hilbert_transform_of_step}
\end{equation}
\end{lem}

\begin{proof}
Denote by $\sqcap_{\alpha}\left(x\right)$ the square pulse of half-width
$\alpha$ around the origin: 
\[
\sqcap_{\alpha}\left(x\right)=\begin{cases}
1 & \abs x<\alpha\\
0 & o.w.
\end{cases}.
\]
Then $u\left(x\right)$ is the sum of infinitely many square pulses,
each with half-width $\pi-\theta_{0}$: 
\[
u\left(x\right)=\sum_{k=-\infty}^{\infty}\sqcap_{\pi-\theta_{0}}\left(x+\pi+2k\pi\right).
\]
The Hilbert transform of a single square pulse $\sqcap_{\alpha}$
can readily be calculated to be
\[
\left(H\sqcap_{\alpha}\right)\left(x\right)=\frac{1}{\pi}\log\abs{\frac{\alpha+x}{\alpha-x}}.
\]
The Hilbert transform commutes with shifts, and so 
\begin{align*}
\left(Hu\right)\left(x\right) & =\frac{1}{\pi}\sum_{k=-\infty}^{\infty}\log\abs{\frac{\pi-\theta_{0}+x+\pi+2k\pi}{\pi-\theta_{0}-\left(x+\pi+2k\pi\right)}}\\
 & =\frac{1}{\pi}\log\prod_{k=-\infty}^{\infty}\abs{1+\frac{\frac{1}{2\pi}\left(2\pi-2\theta_{0}\right)}{k+\frac{1}{2\pi}\left(\theta_{0}+x\right)}}.
\end{align*}
This is an expression of the form form 
\begin{equation}
\prod_{k=-\infty}^{\infty}\abs{1+\frac{b}{k-a}},\label{eq:product_is_sine}
\end{equation}
with $b=\frac{1}{2\pi}\left(2\pi-2\theta_{0}\right)$ and $a=-\frac{1}{2\pi}\left(\theta_{0}+x\right)$.
It can be shown that a product in the form of equation (\ref{eq:product_is_sine})
is equal to 
\[
\frac{\sin\left(\pi\left(a-b\right)\right)}{\sin\left(\pi a\right)}
\]
(to see this, recall Euler's infinite product identity for the sine
function, $\sin\left(x\right)=x\prod_{n=1}^{\infty}\left(1-\frac{x^{2}}{n^{2}\pi^{2}}\right)$)
. We then have
\begin{align*}
\left(Hu\right)\left(x\right) & =\frac{1}{\pi}\log\abs{\frac{\sin\left(\pi\left(-\frac{1}{2\pi}\left(\theta_{0}+x\right)-\frac{1}{2\pi}\left(2\pi-2\theta_{0}\right)\right)\right)}{\sin\left(-\pi\frac{1}{2\pi}\left(\theta_{0}+x\right)\right)}}\\
 & =\frac{1}{\pi}\log\abs{\frac{\sin\left(\frac{\theta_{0}}{2}-\frac{x}{2}\right)}{\sin\left(\frac{\theta_{0}}{2}+\frac{x}{2}\right)}}.
\end{align*}
\end{proof}
Differentiating Equation (\ref{eq:hilbert_transform_of_step}) in
Lemma \ref{lem:hilbert_transform_of_step}, each $y_{i}$ is therefore
given by $\left(H\varphi\right)\left(x\right)$ where $x$ is a zero
of the derivative of $H\varphi$:
\begin{align*}
0 & =\frac{d}{dx}\sum_{i}\frac{\alpha_{i}}{\pi}\log\abs{\frac{\sin\left(\frac{\theta_{i}}{2}-\frac{x}{2}\right)}{\sin\left(\frac{\theta_{i}}{2}+\frac{x}{2}\right)}}\\
 & =\frac{1}{2\pi}\sum_{i}\alpha_{i}\frac{\sin\left(\frac{\theta_{i}}{2}+\frac{x}{2}\right)}{\sin\left(\frac{\theta_{i}}{2}-\frac{x}{2}\right)}\cdot\frac{-\frac{1}{2}\cos\left(\frac{\theta_{i}}{2}-\frac{x}{2}\right)\sin\left(\frac{\theta_{i}}{2}+\frac{x}{2}\right)-\frac{1}{2}\cos\left(\frac{\theta_{i}}{2}+\frac{x}{2}\right)\sin\left(\frac{\theta_{i}}{2}-\frac{x}{2}\right)}{\sin\left(\frac{\theta_{i}}{2}+\frac{x}{2}\right)^{2}}\\
 & =-\frac{1}{4\pi}\sum_{i}\alpha_{i}\left(\cot\left(\frac{\theta_{i}}{2}-\frac{x}{2}\right)+\cot\left(\frac{\theta_{i}}{2}+\frac{x}{2}\right)\right).
\end{align*}
\end{proof}

\subsection{The uniform distribution}

Let $\mu$ be the uniform distribution on $\left[-1,1\right]$, so
that 
\[
F_{\mu}=\begin{cases}
0 & x<-1\\
\frac{1}{2}\left(x+1\right) & x\in\left[-1,1\right]\\
1 & x>1
\end{cases},
\]
the inverse function is $G\left(y\right)=2y-1$, and $\varphi\left(\theta\right)=2\frac{\abs{\theta}}{\pi}-1$.
The Fourier series of $\varphi_{\mu}$ is given by 
\[
\varphi\left(\theta\right)=-\frac{8}{\pi^{2}}\sum_{k=1}^{\infty}\frac{\cos\left(\left(2k-1\right)\theta\right)}{\left(2k-1\right)^{2}},
\]
which gives 
\[
\psi\left(z\right)=-\frac{8}{\pi^{2}}\sum_{k=1}^{\infty}\frac{z^{2k-1}}{\left(2k-1\right)^{2}}.
\]
The boundary $\partial\Omega$ is given in Figure \ref{fig:uniform_distribution_boundary}.
The $y$ component of $\gamma\left(\theta\right)=\psi\left(e^{i\theta}\right)=\left(x\left(\theta\right),y\left(\theta\right)\right)$
is then given by 
\[
y\left(\theta\right)=-\frac{8}{\pi^{2}}\sum_{k=1}^{\infty}\frac{\sin\left(\left(2k-1\right)\theta\right)}{\left(2k-1\right)^{2}}.
\]
Differentiating by $\theta$, we get 
\[
\frac{d}{d\theta}y\left(\theta\right)=-\frac{8}{\pi^{2}}\sum_{k=1}^{\infty}\frac{\cos\left(\left(2k-1\right)\theta\right)}{2k-1}.
\]
If we denote by $\mathrm{gd}^{-1}\left(x\right)=\tanh^{-1}\left(\tan\left(\frac{x}{2}\right)\right)$
the inverse Gudermannian function, then a short calculation reveals
that 
\begin{align*}
\frac{d}{d\theta}y\left(\theta\right) & =-\frac{8}{\pi^{2}}\mathrm{gd}^{-1}\left(\frac{\pi}{2}-\abs{\theta}\right)\\
 & =-\frac{8}{\pi^{2}}\tanh^{-1}\left(\tan\left(\frac{\pi}{4}-\frac{\abs{\theta}}{2}\right)\right).
\end{align*}
Thus the domain $\Omega$ is bounded by the parametric curve
\[
\gamma\left(\theta\right)=\left(\frac{2\abs{\theta}}{\pi}-1,-\int_{-\pi}^{\theta}\frac{8}{\pi^{2}}\tanh^{-1}\left(\tan\left(\frac{\pi}{4}-\frac{\abs{\theta}}{2}\right)\right)\right).
\]
This integral may be solved with the assistance of computational software;
the antiderivative of $\tanh^{-1}\left(\tan\left(\frac{\pi}{4}-\frac{x}{2}\right)\right)$
is given by
\[
x\tanh^{-1}\left(\cot\left(\frac{\pi}{4}-\frac{x}{2}\right)\right)+\frac{1}{2}\left(i\left(\mathrm{Li}_{2}\left(-e^{ix}\right)-\mathrm{Li}_{2}\left(e^{ix}\right)\right)\right)+x\log\left(1-e^{ix}\right)-\log\left(1+e^{ix}\right),
\]
where $\mathrm{Li}_{n}\left(x\right)$ is the polylogarithm function,
\[
\mathrm{Li}_{n}\left(z\right)=\sum_{k=1}^{\infty}\frac{z^{k}}{k^{n}}.
\]

\section{\label{sec:open_questions}Other directions and open questions}

For a given measure $\mu$, the domain $\Omega$ in Theorem \ref{thm:conformal_skorokhod_main_theorem}
is not unique, even if we require $\Omega$ to be simply connected
and symmetric to conjugation (see Figure \ref{fig:non_uniqueness}
for two different domains giving the same distribution). Is there
any distinguishing trait to the construction given above? Can we deduce
properties of $\Omega$ from properties of $\mu$? For example, 
\begin{question}
For what measures $\mu$ is the domain $\Omega$ in Theorem \ref{thm:conformal_skorokhod_main_theorem}
convex?

\begin{figure}
\begin{centering}
\includegraphics[scale=0.4]{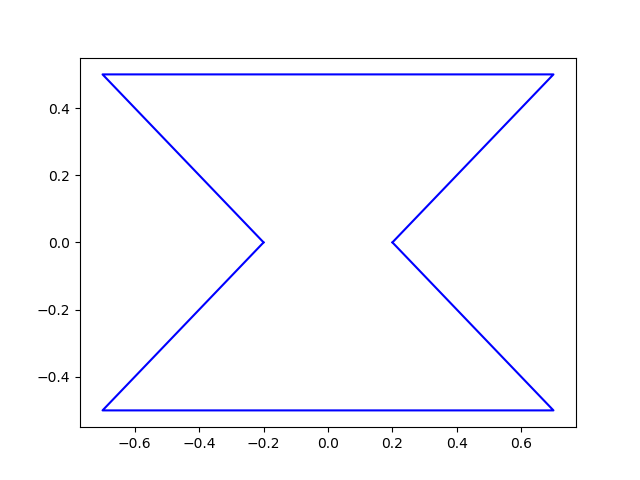}\includegraphics[scale=0.4]{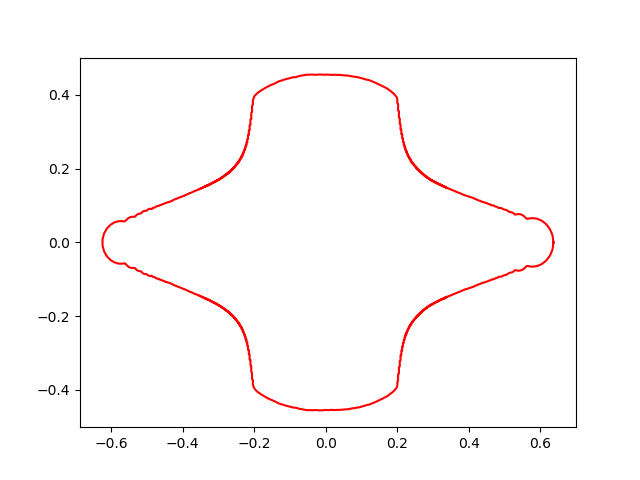}
\par\end{centering}
\caption{Two domains giving rise to the same continuous distribution. The domain
on the right hand side is an approximation obtained by using the empirical
distribution given by the domain on the left hand side.\label{fig:non_uniqueness}}

\end{figure}
\end{question}

Theorem \ref{thm:conformal_skorokhod_main_theorem} utilizes the conformal
invariance of Brownian motion and uses complex-analytic tools, and
so relies on the fact that the Brownian motion is planar. However,
we may ask similar hitting-time questions concerning the marginal
distribution of the first $n$ coordinates of a stopped Brownian motion
in higher dimensions. It is already well known that not every distribution
$\mu$ on $\r^{n}$ can be generated this way: For example, high dimensional
Brownian motion does not hit points, so any $\mu$ containing atoms
cannot be obtained. What can we say about measures $\mu$ for which
we already know that there is a Skorokhod embedding?
\begin{question}
Let $\mu$ be a distribution on $\r^{n}$ for which there exists a
Skorokhod embedding. Is there a domain $\Omega\subseteq\r^{n+m}$
for some $m$ so that $\left(Y_{T\left(X_{t},\Omega\right)}^{\left(1\right)},\ldots,Y_{T\left(X_{t},\Omega\right)}^{\left(n\right)}\right)$
is a Skorokhod embedding for $\mu$?
\end{question}

\section{Acknowledgments}

The author thanks Itai Benjamini, Krzysztof Burdzy, Ronen Eldan, Boa'z
Klartag and Dan Mikulincer for their insightful comments and suggestions.

\bibliographystyle{plain}
\bibliography{conformal_embedding_bibliography}

\end{document}